\documentclass[a4paper,11 pt]{amsart}

\RequirePackage{amsmath, amssymb, amsthm, amsfonts}
\usepackage[all]{xy}
\usepackage{enumitem}
\usepackage{hyperref}
\usepackage{fullpage}
\usepackage{hyperref}
\usepackage[mathscr]{eucal}
\usepackage[active]{srcltx}
\newif\ifpdf
\ifx\pdfoutput\undefined \pdffalse \else \ifnum\pdfoutput=0
\pdffalse \else \pdfoutput=1 \pdftrue \fi \fi

         \ifpdf
         \usepackage[pdftex]{graphicx}
         \else
         \usepackage{graphicx}
         \fi
\newtheorem{theo}{Theorem}[section]
\newtheorem{lem}[theo]{Lemma}
\newtheorem{prop}[theo]{Proposition}
\newtheorem{cor}[theo]{Corollary}
\newtheorem{defin}[theo]{Definition}
\newtheorem{notation}[theo]{Notation}
\theoremstyle{definition}

\newtheorem{rem}[theo]{Remark}

\newcommand{\Aut}{{\mathrm{Aut}}}
\newcommand{\Out}{{\mathrm{Out}}}
\newcommand{\Inn}{{\mathrm{Inn}}}

\newcommand{\Diff}{{\mathrm{Diff}}}
\newcommand{\Map}{{\mathrm{Map}}}

\newcommand{\ZZ}{\mathbb{Z}}

\newcommand{\lr}{\longrightarrow}

\title{Surfaces isogenous to a product of curves, braid groups and mapping class groups}
\author{Matteo Penegini}
\address{Matteo Penegini\\
Dipartimento di Matematica ``F. Enriques" \\ Via Saldini 50, I-20133
Milano, Italy} \email{matteo.penegini@unimi.it}

\begin{document}


\maketitle

\tableofcontents


\section{Introduction}\label{sec.intro}

This article is a revised version of the talk I gave at the
conference ``\textit{Beauville Surfaces and groups}'' held in
Newcastle in June 2012. It presents some group theoretical
methods to give bounds on the number of connected components of
the moduli space of surfaces of general type, focusing on some
families of regular surfaces isogenous to a product of curves.
Some of the results appearing in this work have been proven
in collaboration with Shelly Garion.

We will use the standard notation from the theory of complex
algebraic surfaces. Let $S$ be a smooth, complex, projective,
\emph{minimal} surface $S$ of \emph{general type}; this means that
the canonical divisor $K_S$ of $S$ is \emph{big} and
\emph{nef}.



The principal numerical invariants for the study of minimal
surfaces of general type are

\begin{itemize}
    \item the \emph{geometric genus} $p_g(S):=h^0(S, \,
\Omega^2_S)=h^0(S, \, \mathcal{O}_S(K_S))$,
    \item the
\emph{irregularity} $q(S):=h^0(S, \, \Omega^1_S)$, and
    \item the \emph{self
intersection of the canonical divisor} $K^2_S$.
\end{itemize}

In fact, these determine all the other classical
invariants, as
\begin{itemize}
   \item the
\emph{Euler-Poincar\'e characteristic}
$\chi(\mathcal{O}_S)=1-q(S)+p_g(S)$,
    \item the \emph{topological Euler number}
    $e(S)=12\chi(S)-K^2_S$, and
    \item the \emph{plurigenera} $P_n(S)=\chi(S)+(^n_2)K^2_S$.
\end{itemize}

Moreover, we call a surface \emph{regular} if its irregularity vanishes, i.e., $q(S)=0$.

By a theorem of Bombieri, a minimal surface of general type $S$
with fixed invariants is birationally mapped to a normal surface
$X$ in a fixed projective space of dimension $P_5(S)-1$. Moreover,
$X$ is uniquely determined and is called the \emph{canonical
model} of $S$. Let us recall Gieseker's Theorem.

\begin{theo}
 There exists a quasi-projective coarse moduli space
$\mathcal{M}_{y, x}$ for canonical models of surfaces of
general type $S$ with fixed invariants $y:=K^2_S$ and $x:=\chi$.
\end{theo}

The union $\mathcal{M}$ over all admissible pairs of invariants ($y,x$) of
these spaces is called the \emph{moduli space of surfaces of
general type}.
If $S$ is a smooth minimal surface of general type, we denote by
$\mathcal{M}(S)$ the subvariety of $\mathcal{M}_{y,x}$,
corresponding to surfaces (orientedly) homeomorphic to $S$.
Moreover, we denote by $\mathcal{M}^0_{y,x}$ the subspace of
the moduli space corresponding to regular surfaces. 

It is known that
the number of connected components $\delta(y, x)$ of
$\mathcal{M}^0_{y,x}$ is bounded from above by a function of $y$; more precisely by \cite{Cat92} we have
$\delta(y, x) \leq cy^{77y^2}$, where $c$ is a positive constant. Hence we have that the number of
components has an exponential upper bound in $K^2$.

There are also some results regarding the lower bound. In
\cite{Man}, for example, Manetti constructed a sequence $S_n$ of
simply connected surfaces of general type, such that the lower
bound for the number of the connected components  $\delta(S_n)$ of
$\mathcal{M}(S_n)$ is given by
\[
\delta(S_n) \geq y_n^{\frac{1}{5}log y_n}.
\]

Using group theoretical methods we are able to describe the asymptotic growth of
the number of connected components of the moduli space of surfaces
of general type relative to certain sequences of surfaces. More
precisely, we apply the definition and some properties of regular
surfaces isogenous to a product of curves and of some special
cases of them, Beauville surfaces, to reduce the geometric problem
of finding connected components into the algebraic one of counting
orbits of some group action, which can be effectively computed.

The paper is organized as follows.

In the first Section we recall the definition and some properties
of the mapping class group. Moreover, we describe the Hurwitz moves
in the most general setting.

In the second Section we briefly recall the definition of surfaces
isogenous to a product of curves. We will recall some of the
properties of these surfaces focusing on their moduli space.

In the third part we  present some results obtained with Shelly
Garion about the number of connected components of the moduli
space of surfaces isogenous to a product.

Finally, we also point out some possible future
developments.

\medskip

 \textbf{Acknowledgments.} The author is grateful to G. Bini for reading and commenting the paper. Moreover the author thanks the organizers of the conference \textit{Beauville surfaces and Groups} N. Barker, I. Bauer, S. Garion and A. Vdovina for the invitation and the kind hospitality.
 

\section{Braid Group and Mapping Class Group}\label{sec.BraidGroup}

In this section, we first recall the definition of mapping class group. Next
we give a presentation of it for
$\mathbb{P}^1-\{p_1,\dots,p_r\}$, and more generally for a curve of genus $g'$
with $r$ marked points. After that, we calculate the Hurwitz moves induced by
those groups. We mainly follow the definitions and notation
of \cite{catdd}.

\begin{defin} Let $M$ be a differentiable manifold, then the \emph{mapping
class group}\index{Mapping Class Group} of $M$ is the group:
\[ \Map(M):= \pi_0(\Diff^+(M)):=\Diff^+(M)/\Diff^0(M),
\]
where $\Diff^+(M)$ is the group of orientation preserving
diffeomorphisms of $M$ and $\Diff^0(M)$ is the subgroup of
diffeomorphisms of $M$ isotopic to the identity.
\\ If $M$ is a compact complex curve of genus $g'$ we use the following notation:
\begin{enumerate}
    \item We denote the mapping class group of $M$ without marked points by
    $\Map_{g'}$.
    \item If we consider $r$ unordered marked points $p_1,\dots, p_r$ on $M$ we define:
\[ \Map_{g',[r]}=\pi_0(\Diff^+(M-\{p_1,\dots,p_r\})),
\]
and this is known as the \emph{full mapping class group}.
\end{enumerate}
\end{defin}
There is a way to present the full mapping class group of a curve using
three different types of twists.
\begin{theo}\label{thm_braid} The mapping class group
$\Map_{0,[r]} = \pi_0(\Diff^+(\mathbb{P}^1-\{p_1,\dots,p_r\}))$ is isomorphic to the \emph{braid group} $\mathbf{B}_r$  on $r$ strands, which can
be presented as
\[
    \mathbf{B}_r = \langle \sigma_1,\dots,\sigma_{r-1} | \sigma_i \sigma_{i+1}
    \sigma_i = \sigma_{i+1} \sigma_i\sigma_{i+1}, \sigma_i \sigma_j =
    \sigma_j \sigma_i \text{ if } |i-j|\geq 2 \rangle.
\]
\end{theo}
For a proof of the above Theorem see, for example, \cite[Theorem 1.11]{Bir}.

In this way Artin's standard generators $\sigma_i$ ($i=1, \ldots ,r-1$) of $\mathbf{B}_r$  can be
represented by the so-called half-twists.

\begin{defin} The \emph{half-twist} $\sigma_j$ is a diffeomorphism of
$\mathbb{P}^1-\{p_1,\dots,p_r\}$ isotopic to the homeomorphism given by (see Figure 1):
\begin{itemize}
    \item A rotation of $180$ degrees on the disk with center
    $j+\frac{1}{2}$ and radius $\frac{1}{2}$;
    \item on a circle with the same center and radius
    $\frac{2+t}{4}$ the map $\sigma_j$ is the identity if $t \geq
    1$ and a rotation of $180(1-t)$ degrees, if $t \leq 1$.
\end{itemize}
\end{defin}

\begin{center}
\includegraphics[totalheight=5 cm, width=15 cm]{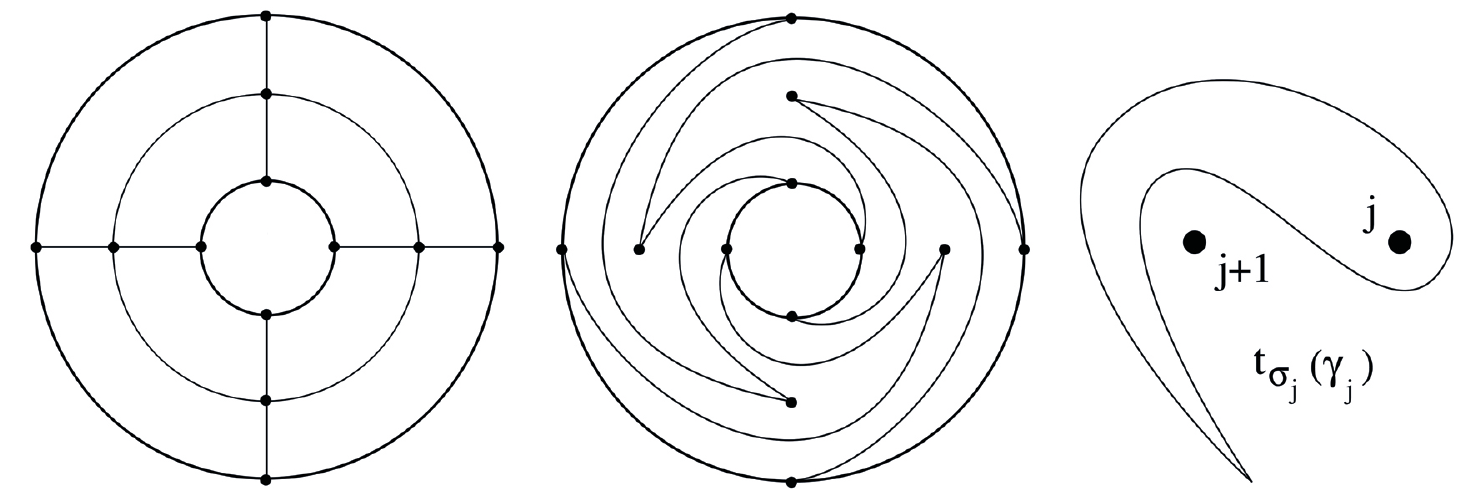} \\
Figure 1.
\end{center}


We want to give a similar presentation for a group $\Map_{g'}$
with $g' \geq 1$, so we have to introduce the Dehn twists.
\begin{defin} Let $C$ be an oriented Riemann surface. Then a
\emph{positive Dehn twist} $t_{\alpha}$ with respect to a simple closed curve
$\alpha$ on $C$ is an isotopy class of a diffeomorphism $h$ of $C$
which is equal to the identity outside a neighborhood of $\alpha$
orientedly homeomorphic to an annulus in the plane, while inside
the annulus $h$ rotates the inner boundary of the annulus by
$360^{\circ}$ to the right and damps the rotation down to the
identity at the outer boundary (see Figure 2).
\end{defin}

\begin{center}
\includegraphics[totalheight=5 cm, width=15 cm]{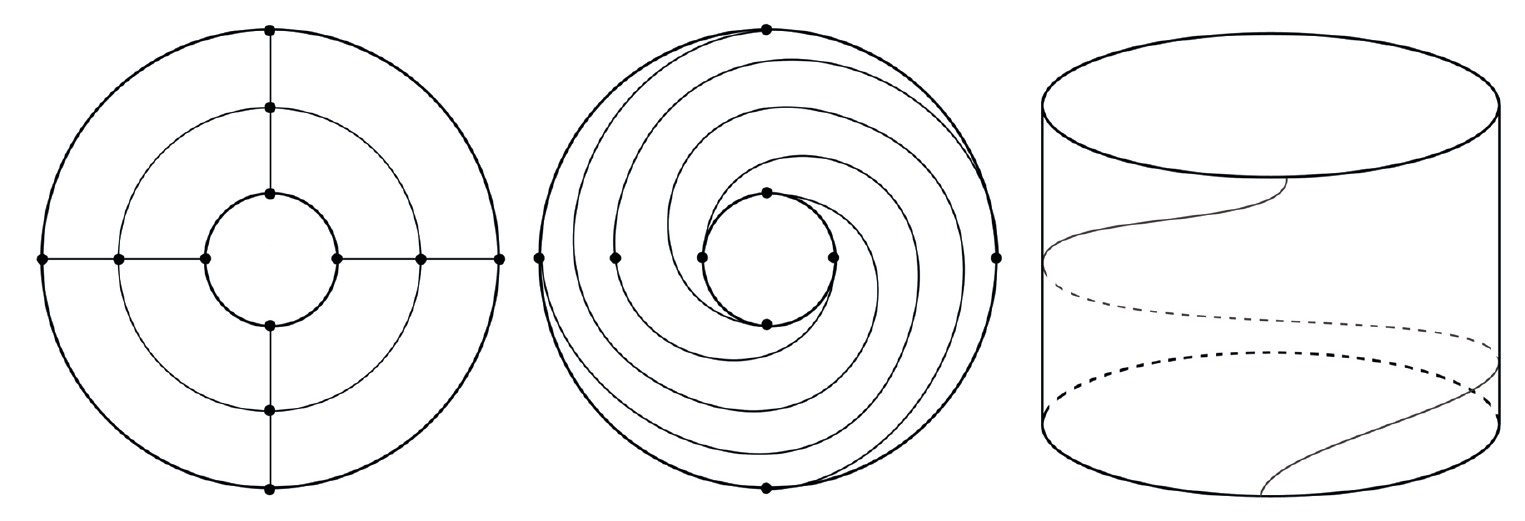} \\
Figure 2.
\end{center}

We have then the following classical results of Dehn \cite{deh}.
\begin{theo} The mapping class group $\Map_{g'}$ is
generated by Dehn twists.
\end{theo}
We give the generators of the group $\Map_{g'}$.
\begin{theo} The group $\Map_{g'}$ is generated by the Dehn
twists with respect to the curves in the Figure 3.
\begin{center}
\includegraphics[totalheight=5 cm, width=15 cm]{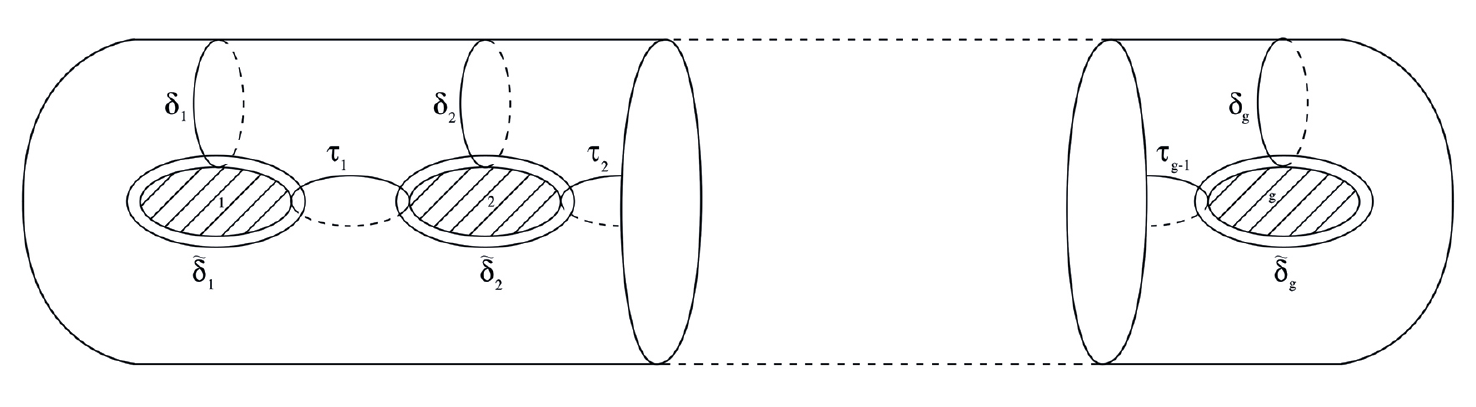} \\
{\rm Figure 3}.
\end{center}
\end{theo}

A proof of the above Theorem can be found in \cite[Theorem
4.8]{Bir}.

For a more general situation where the surface has $g'>0$ and $r$
marked points we need to introduce a third type of twist: The
$\xi$-twists, which link the holes with the marked points. Let us
recall the Birman short exact sequence for a Riemann surface $C'$
with $g(C')>1$ or $r>1$, setting $\mathcal{B}:=\{p_1, \ldots ,p_r\}$:
\begin{equation}\label{eq.Birman} 1 \longrightarrow \pi_1(C' \setminus \mathcal{B},p)
\stackrel{\Xi}{\longrightarrow} \pi_0(\Diff^+(C'\setminus (\mathcal{B} \ \cup\{p\}))
\longrightarrow \pi_0(\Diff^+(C'\setminus \mathcal{B}) \longrightarrow 1.
 \end{equation}
 The map $\Xi$ can be described as follows (cf. \cite{B}). Let $[\gamma] \in
 \pi_1(C' \setminus \mathcal{B},p)$ and $\gamma$ be a simple, smooth loop based at $p$ representing $[\gamma]$. Then $\Xi([\gamma])$ is the isotopy class of a
 $\xi$-twist with respect to the closed curve $\gamma$. In addition, this new twist is isotopic to the identity in
$\pi_0(\Diff^+(C\setminus (\mathcal{B}))$.
 Let us now describe a $\xi$-twist. We shall
consider the annulus $A:=\{z=\rho e^{i \theta} \in \mathbb{C} | 1
\leq \rho \leq 2\}$, and we define $h: A \rightarrow A$ as follows
\begin{equation}
h(\rho, \theta): \left\{
\begin{array}{rl}
\big(\rho, \theta - 4\pi(\rho - 1)\big) & 1 \leq \rho \leq \frac{3}{2}  \\
\big(\rho, \theta - 4\pi(2-\rho)\big) & \frac{3}{2} \leq \rho \leq 2. \\

\end{array}
\right.
\end{equation}
\begin{defin}
Let $C$ be a Riemann surface, and $\alpha$ a simple closed curve
on $C$. Let $\iota$ be a diffeomorphism between $A$ and a tubular
neighborhood of $\alpha$. Then the \emph{$\xi$-twist}
$\mathbf{t}_{\alpha}$ with respect to $\alpha$ is defined as $\iota \circ
h \circ \iota^{-1} \mid_{\iota(A)}$ extended to the whole $C$ as
the identity on $C \setminus \iota(A)$ (see Figure 4).
\end{defin}
\begin{center}
\includegraphics[totalheight=5 cm, width=15 cm]{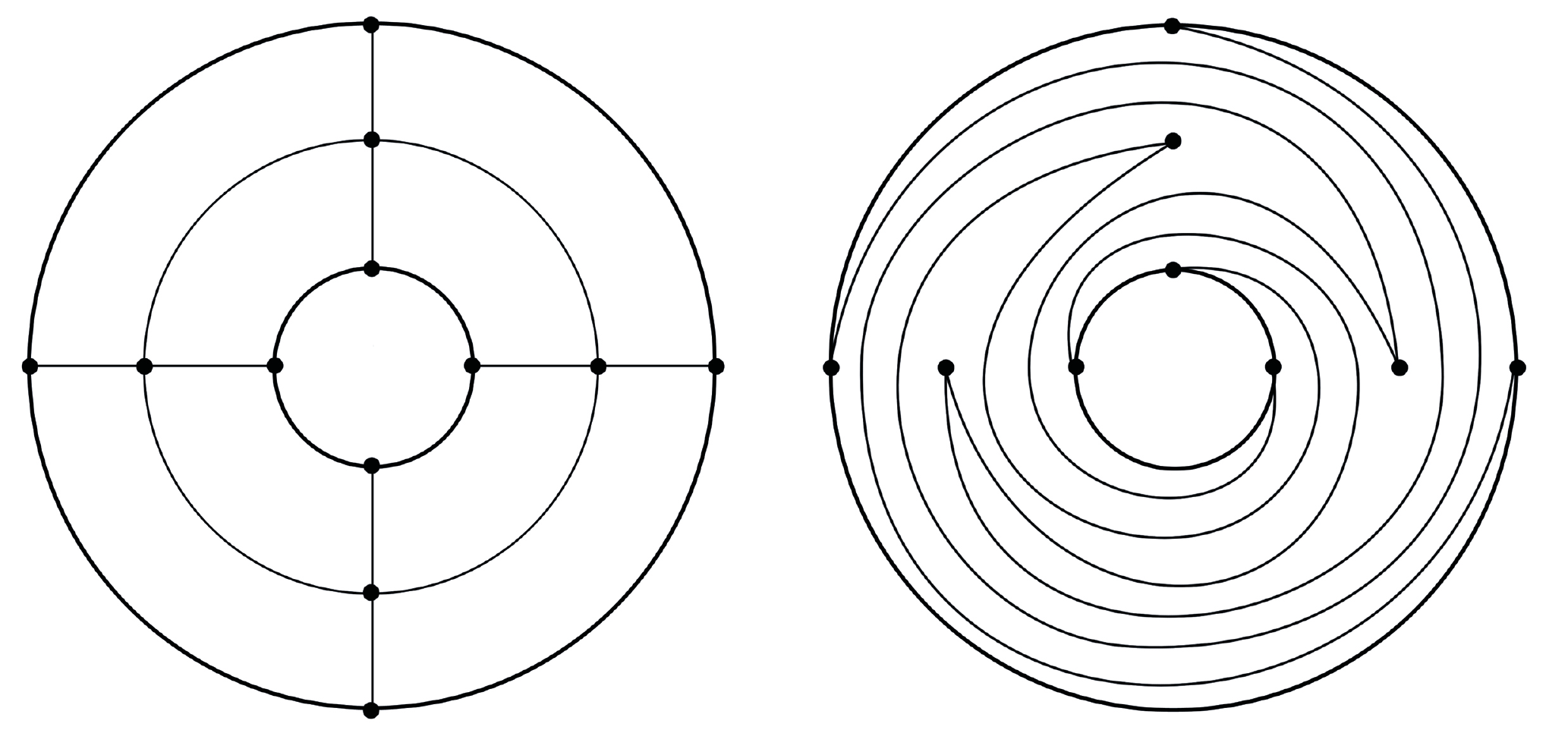} \\
Figure 4.
\end{center}

Therefore, for the more general situation we have the following.
\begin{theo}\cite[Theorem 3]{B} Let $g(C') \neq 0$ and $g(C')>1$ or $r>1$ then the group $\Map_{g',[r]}$ is generated by the $3g'-1$ Dehn twists with respect
to the curves $\delta_j$, $\tilde{\delta_j}$ and $\tau_j$, by the $2rg'$ $\xi$-twists with respect to the curves
$\xi^l_{j,d}$ and the $r-1$ half-twists about the points $p_1, \ldots
,p_r$ in {\rm Figure 5}.
\end{theo}
\begin{center}
\includegraphics[totalheight=7 cm, width=18 cm]{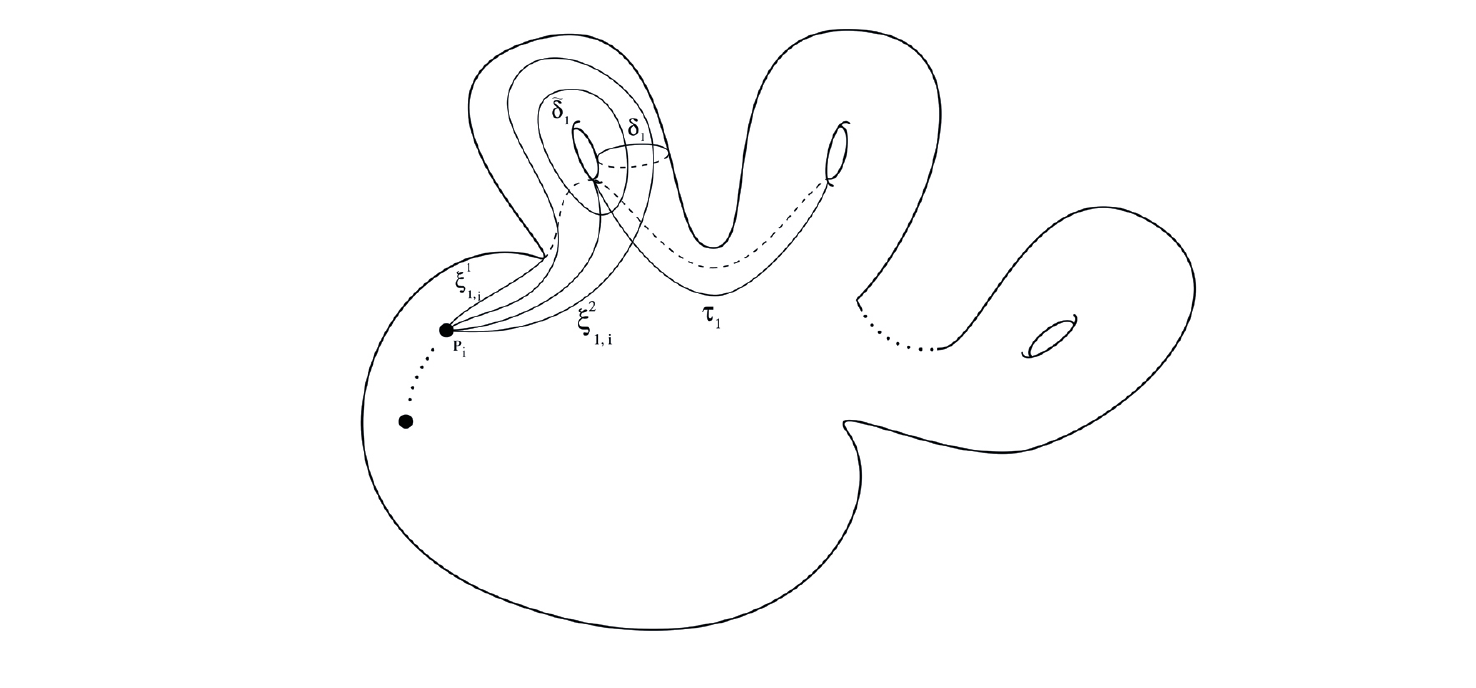} \\
Figure 5.
\end{center}

Let $C'$ be a Riemann surfaces of genus $g'$, and let
$\mathcal{B}:=\{ p_1, \ldots ,p_r\}$ a set of points on $C'$. A
\emph{geometric basis} of $\pi_1(C'\setminus \mathcal{B}, p_0)$
consists of simple non-intersecting (away from the base point)
loops (see Figure 6)
$$
\gamma_1 , \dots , \gamma_d , \alpha_ 1, \beta_1, \dots ,
\alpha_{g'}, \beta_{g'}
$$
such that we get the presentation

\begin{equation*}
 \pi_1(C \setminus \mathcal{B}, p_0):=\langle \alpha_1,\beta_1, \ldots , \alpha_{g'},\beta_{g'},\gamma_{1},
\ldots , \gamma_{r} | \gamma_{1} \cdot \ldots \cdot
\gamma_{r}\cdot\prod_{k=1}^{g'} [\alpha_{k},\beta_{k}]
 =1 \rangle.
\end{equation*}

\begin{center}
\includegraphics[totalheight=7 cm, width=18 cm]{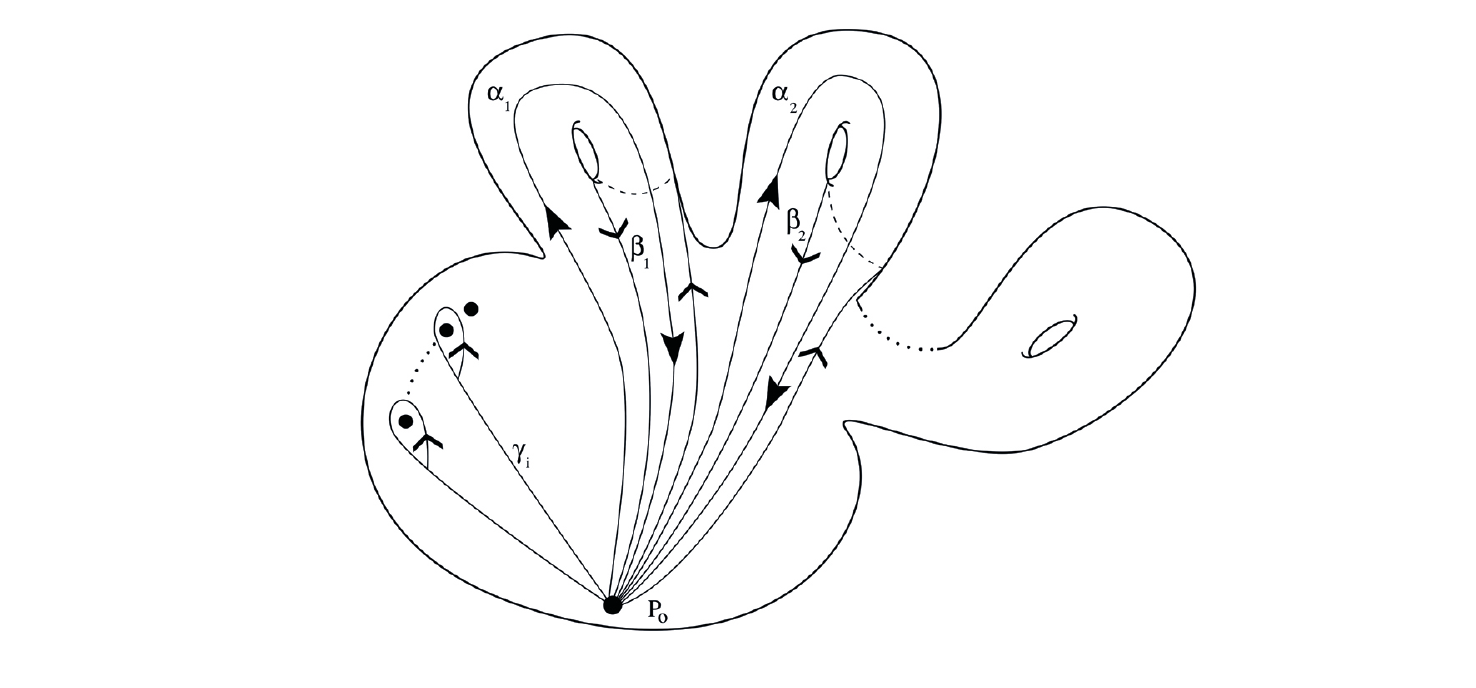} \\
Figure 6.
\end{center}

\begin{defin}\label{defn.orbifold.fuchsian}
Let $g',m_1, \dots , m_r$ be positive integers. An \emph{orbifold
surface group} of type $(g' \mid m_1, \dots , m_r)$ is a group
presented as follows:
\begin{multline*}  \Gamma(g' \mid m_1, \dots , m_r):=\langle a_1,b_1, \ldots , a_{g'},b_{g'},c_{1},
\ldots , c_{r} | \\
c^{m_1}_{1}=\dots=c^{m_r}_{r}=c_{1} \cdot \ldots \cdot
c_{r}\cdot\prod_{k=1}^{g'} [a_{k},b_{k}] =1 \rangle.
\end{multline*}
\end{defin}
We notice that the choice of a geometric basis yields an obvious
epimorphism \\ $ \pi_1(C \setminus \mathcal{B}, p_0) \rightarrow
\Gamma(g' \mid m_1, \dots , m_r)$.

The following is a reformulation of {\em Riemann's existence theorem}:

\begin{theo} A finite group $G$ acts as a group of automorphisms on some
compact Riemann
surface $C$ of genus $g$ if and only if there are natural numbers
$g', m_1, \ldots ,
m_r$, and an orbifold homomorphism
\begin{equation}\label{eq.theta}
\theta \colon \Gamma(g' \mid m_1, \dots , m_r)
\rightarrow G
\end{equation}
such that
$ord(\theta(c_i))=m_i$ for all $i$ and such that the Riemann - Hurwitz relation holds:
\begin{equation}\label{eq.RiemHurw} 2g - 2 = |G|\left(2g'-2 + \sum_{i=1}^r \left(1 -
\frac{1}{m_i}\right)\right).
\end{equation}

\end{theo}

If this is the case, then $g'$ is the genus of $C':=C/G$. The
$G$-cover $C \rightarrow C'$ is branched at $r$ points $p_1,
\ldots , p_r$ with branching indices $m_1, \ldots , m_r$,
respectively. Let $\Gamma=\Gamma(g' \mid m_1,...,m_r)$ be an
orbifold surface group with a presentation as in Definition
\ref{defn.orbifold.fuchsian}. If $G$ is a finite group quotient of
$\Gamma$ as in \eqref{eq.theta}, then we say that $G$ is $(g' \mid
m_1,\dots,m_r)-$generated, the image of the generators of $\Gamma$
in $G$ is called a \emph{system of generators} for $G$. Finally,
$\theta$ is called an \emph{admissible epimorphism}.
\begin{defin} An automorphism $\eta \in
\Aut(\Gamma)$ is said to be \emph{orientation preserving} if the action
induced on $<\alpha_1,\beta_1, \cdots,
\alpha_{g'},\beta_{g'}>^{ab}$ has determinant $+1$, and for all
$i\in \{1,\dots,r\}$ there exists $j$ such that $\eta(\gamma_i)$
is conjugate to $\gamma_j$, which implies
$ord(\gamma_i)=ord(\gamma_j)$.

The
subgroup of orientation preserving automorphisms of $\Gamma$ is
denoted by $\Aut^+(\Gamma)$ and the quotient
$\Out^+(\Gamma):=\Aut^+(\Gamma)/\Inn(\Gamma)$ is called the
\emph{mapping class group} of $\Gamma$.
\end{defin}

\begin{theo}\label{out}
Let $\Gamma=\Gamma(g' \mid m_1,\dots,m_r)$ be
an orbifold surface group. Then there is an isomorphism of groups:
\[ \Out^+(\Gamma) \cong \Map_{g',[r]}.
\]
\end{theo}
This is a classical result cf. e.g., \cite[\S 4]{Macl} .

Moreover let $G$ be a finite group $(g' \mid
m_1,\dots,m_r)-$generated. There is a section $s: \Out^+(\Gamma)
\rightarrow \Aut^+(\Gamma)$, which induces an action of the
$\Map_{g',[r]}$ on the generators of $\Gamma$. Such action does
not depend on $s$ up to simultaneous conjugation, meaning that the
action is defined up to inner automorphisms. This action induces
an action on the systems of generators of $G$ via composition with
admissible epimorphisms.

\begin{defin}\label{humov} Let $G$ be a finite group $(g'
\mid m_1,\dots,m_r)-$generated. If two systems of generators
$\mathcal{V}_1$ and $\mathcal{V}_2$ are in the same
$\Map_{g',[r]}$-orbit, we say that they are related by a
\emph{Hurwitz move}  (or are \emph{Hurwitz equivalent}).
\end{defin}

\begin{prop}\label{prop_Hurwitzmoves}
Let $C'$ be a curve of genus $g'$, $\mathcal{B}:=\{p_1,\ldots,p_r\}$, and with $g' \neq 0$ and $g'>1$ or $r>1$.
Up to inner automorphisms, the action of $\Map_{g',[r]}$ on
$\Gamma(g'\;| \; m_1 \ldots m_r)$ is induced by the following action on a geometric basis of $\pi_1(C'\setminus \mathcal{B}, p_0)$

\[
\mathbf{t_{\delta_j}}: \left\{
\begin{array}{rll}
\alpha_j & \mapsto \alpha_j\beta^{-1}_j & \\
\alpha_i & \mapsto \alpha_i             & \forall i \neq j \\
\beta_i  & \mapsto  \beta_i             & \forall i\\
\gamma_i & \mapsto \gamma_i             & \forall i\\
\end{array}
\right.
\mbox{   }
\mathbf{t_{\tilde{\delta_j}}}: \left\{
\begin{array}{rll}
\alpha_i & \mapsto\alpha_i           & \forall i \\
\beta_j  & \mapsto  \beta_j \alpha_j & \\
\beta_i  & \mapsto \beta_i           & \forall i \neq j\\
\gamma_i & \mapsto\gamma_i           & \forall i\\
\end{array}
\right.
\]
\[
\mathbf{t_{\sigma_h}}: \left\{
\begin{array}{rll}
\alpha_i & \mapsto \alpha_i                      & \forall i\\
\beta_i  & \mapsto  \beta_i                      & \forall i\\
\gamma_h & \mapsto \gamma_{h+1}                  & \\
\gamma_{h+1} & \mapsto \gamma^{-1}_{h+1}\gamma_h\gamma_{h+1}   & \\
\gamma_i & \mapsto\gamma_i           & \forall i \neq h, h+1\\
\end{array}
\right.
\mbox{   }
\mathbf{t_{\tau_k}}: \left\{
\begin{array}{rll}
\alpha_k & \mapsto\alpha_k \eta^{-1}_{k}          & \\
\beta_k  & \mapsto  \beta_k^{\eta_k}              & \\
\alpha_{k+1} & \mapsto \eta_{k}\alpha_k           & \\
\alpha_i & \mapsto\alpha_i           & \forall i \neq k,k+1 \\
\beta_i  & \mapsto \beta_i           & \forall i \neq k\\
\gamma_i & \mapsto\gamma_i           & \forall i\\
\end{array}
\right.
\]
\[
\mathbf{t}_{\xi^1_{j,d}}: \left\{
\begin{array}{rll}
\alpha_j & \mapsto \chi_{j,d}\alpha_j & \\
\alpha_i & \mapsto \alpha_i         & \forall i \neq j \\
\beta_i  & \mapsto \beta_i           & \forall i\\
\gamma_d  & \mapsto     \gamma^{\epsilon_{j,d}}_d &\\
\gamma_i & \mapsto\gamma_i           & \forall j \neq d\\
\end{array}
\right.
\mbox{   }
\mathbf{t}_{\xi^2_{j,d}}: \left\{
\begin{array}{rll}
\alpha_i & \mapsto \alpha_i         & \forall i  \\
\beta_j  & \mapsto  \alpha^{-1}_j\chi_{j,d}\alpha_j\beta_j & \\
\beta_i  & \mapsto \beta_i           & \forall i \neq j\\
\gamma_d  & \mapsto     \gamma^{\epsilon'_{j,d}}_d &\\
\gamma_i & \mapsto\gamma_i           & \forall j \neq d\\
\end{array}
\right.
\]
for $1 \leq j \leq g'$, $1 \leq k \leq (g'-1)$, $1 \leq h \leq (r-1)$, and $1 \leq d \leq r$.
Moreover we set $\eta_k:=\beta^{-1}_k\alpha_{k+1}\beta_{k+1}\alpha^{-1}_{k+1}$, $\chi_{j,d}:=(\Pi^{j-1}_{k=1}[\alpha_k, \beta_k])^{-1} \gamma_d         \Pi^{j-1}_{k=1}[\alpha_k, \beta_k]$, \\
$\epsilon_{j,d}:=\gamma_d(\Pi^{j}_{k=1}[\alpha_k,
\beta_k])\alpha_j\beta_j \alpha^{-1}_j(\Pi^{j}_{k=1}[\alpha_k,
\beta_k])^{-1}$, and
$\epsilon'_{j,d}:=\gamma_d(\Pi^{j}_{k=1}[\alpha_k, \beta_k])
\alpha^{-1}_j(\Pi^{j}_{k=1}[\alpha_k, \beta_k])^{-1}$.
\end{prop}
In the above proposition the twists $\mathbf{t}_{\tilde{\delta_j}},\mathbf{t}_{\delta_j}$ and $ \mathbf{t}_{\tau_j}$ correspond to Dehn twists, $\mathbf{t}_{\mathbf{\xi}_{r_{j,d}}}$ and $\mathbf{t}_{\mathbf{\xi}_{r_{j,d}}}$ to $\xi$-twists, and finally $\mathbf{t}_{\sigma_h}$ to half-twists.

\proof One notices that a Riemann surface of genus $g'$ is a
connected sum of $g'$ tori. Then one can use the Figure 7 to
calculate the Dehn twists about the curves $\delta_j$, and similarly for the Dehn twists about the curves $\tilde{\delta_j}$ . One can use the results given in \cite{P11} to calculate the
Dehn twists about the curves $\tau_j$. In the Appendix of \cite{CLP} are
described the actions of the $\xi$-twists. Finally
the half-twists action is clear by Figure 1.
\endproof

\begin{center}
\includegraphics[totalheight=5 cm, width=15 cm]{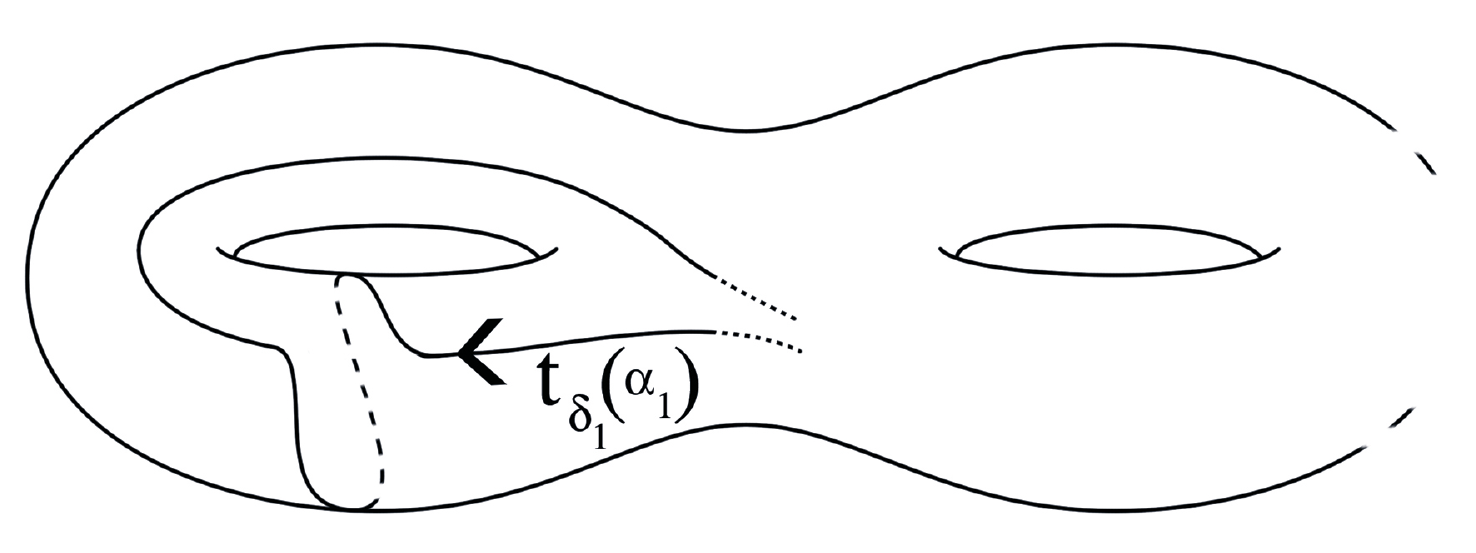} \\
Figure 7.
\end{center}

For the case $g(C')=1$ and $r=1$ see e.g., \cite{pol1} for a proof of the following proposition.
\begin{prop}\cite[Proposition 1.10]{pol1} \label{mong1n1}
Up to inner automorphisms, the action of $\Map_{1,1}$ on
$\Gamma(1\;| \; m_1)$ is induced by the following action on a geometric basis of $\pi_1(C'\setminus \mathcal{B}, p_0)$
\[
\mathbf{t_{\tilde{\delta}}}: \left\{
\begin{array}{rl}
\alpha_1 & \rightarrow \alpha_1  \\
\beta_1  & \rightarrow  \beta_1 \alpha_1 \\
\gamma_1 & \rightarrow \gamma_1 \\
\end{array}
\right.
\mbox{   }
\mathbf{t_{\delta}}: \left\{
\begin{array}{rl}
\alpha_1 & \rightarrow \alpha_1\beta^{-1}_1  \\
\beta_1  & \rightarrow  \beta_1 \\
\gamma_1 & \rightarrow \gamma_1. \\
\end{array}
\right.
\]
\end{prop}

Finally we have.

\begin{prop}\label{prop braidaction} Up to inner automorphism, the action of
$\Map_{0,[r]}$ on $\Gamma (0 \mid m_1, \ldots, m_r)$ is induced by the following action on a geometric basis of $\pi_1(\mathbb{P}^1\setminus \mathcal{B}, p_0)$
\[
\sigma_i: \left\{
\begin{array}{rl} \gamma_i & \rightarrow \gamma_{i+1} \\
\gamma_{i+1}  & \rightarrow \gamma^{-1}_{i+1}\gamma_i\gamma_{i+1} \\
\gamma_j  & \rightarrow \gamma_{j} \textrm{ if } j \neq i, i+1,
\end{array}
\right.
\]
for $i=1, \ldots, r-1$.
\end{prop}
%
%
%
%
%

\section{Surfaces Isogenous to a Product of Curves}\label{sec.SurfIsoProd}

A surface $S$ is said to be \emph{isogenous to a (higher) product of
curves} if and only if $S$ is a quotient $S:=(C_1 \times C_2)/G$, where
$C_1$ and $C_2$ are curves of genus at least two, and $G$ is a finite
group acting freely on $C_1 \times C_2$.

Let $S$ be a surface isogenous to a higher product, and
$G^{\circ}:=G \cap(Aut(C_1) \times Aut(C_2))$. Then $G^{\circ}$ acts
on the two factors $C_1$ and $C_2$ and diagonally on the product $C_1
\times C_2$. If $G^{\circ}$ acts faithfully on both curves, we say
that $S= (C_1 \times C_2)/G$ is a \emph{minimal realization} of  $S$.
In \cite{cat00}, the author proves that any surface isogenous to a higher
product admits a unique minimal realization. From now on, we
work only with minimal realizations.

There are two cases: the \emph{mixed} case where the action of $G$
exchanges the two factors (in this case $C_1$ and $C_2$ are isomorphic
and $G^{\circ} \neq G$); the \emph{unmixed} case (where
$G=G^{\circ}$, and therefore it acts diagonally).

Moreover, we observe that a surface isogenous to a product
of curves is of general type. It is always minimal and its
numerical invariants are explicitly given in terms of the genera
of the curves and the order of the group. Indeed, we have the
following proposition.
\begin{prop} Let $S=(C_1 \times C_2)/G$ be a surface isogenous to a higher product of curves, then:
\begin{equation}\label{eq.chi.isot.fib}
\chi(S)=\frac{(g(C_1)-1)(g(C_2)-1)}{|G|},
\quad
e(S)=4 \chi(S),
\quad
K^2_S=8 \chi(S).
\end{equation}
The irregularity of these surfaces is easily computed by
\begin{equation} q(S)=g(C_1/G)+g(C_2/G).
\end{equation}
\end{prop}

By the above formula a surface $S$ isogenous to a product of
curves has $q(S)=0$ if and only if the two quotients $C_i/G$ are
isomorphic to $\mathbb{P}^1$. Moreover, if both coverings $C_i
\rightarrow C_i/G\cong \mathbb{P}^1$ are ramified in exactly $3$
points, $S$ is a \emph{Beauville surface}. This last condition
is equivalent to saying that Beauville surfaces are rigid, i.e.,
have no nontrivial deformations.

In the unmixed case  $G$ acts separately on $C_1$ and $C_2$, and
the two projections $\pi_i \colon C_1 \times C_2 \lr C_i$ for
$i=1,2$ induce two isotrivial fibrations $\alpha_i \colon S \lr
C_i/G$ for $i=1,2$, whose smooth fibres are isomorphic to $C_2$
and $C_1$, respectively. \emph{We work only with surfaces of
unmixed type}.

Working out the definition of surfaces isogenous to a product, one
sees that there is a pure group theoretical condition which
characterizes the groups of such surfaces: the existence of a
''\emph{ramification structure}''.

\begin{defin}\label{defn.sphergen}
Let $G$ be a finite group and  $\theta_1 \colon \Gamma(g'_1 \mid
m_{1,1}, \dots , m_{1,r_1}) \twoheadrightarrow G$ an admissible
epimorphism. Let $\mathcal{V}_1$ be the system of generators of
$G$ induced by $\theta_1$, i.e., the elements of $G$ which are
images of the generators of $\Gamma$. We say that $\mathcal{V}_1$
is of \emph{type} $\tau_1:=(g'_1 \mid m_{1,1}, \dots ,
m_{1,r_1})$.

Moreover, let $\theta_2 \colon \Gamma(g'_2 \mid m_{2,1}, \dots ,
m_{2,r_2}) \twoheadrightarrow G$ be another admissible epimorphism
and  $\mathcal{V}_2$ be the system of generators of $G$ induced by
$\theta_2$. Then $\mathcal{V}_1$ and $\mathcal{V}_2$ are said to
be \emph{disjoint}, if:
\begin{equation}\label{eq.sigmasetcond} \Sigma(\mathcal{V}_1)
\bigcap \Sigma(\mathcal{V}_2)= \{ 1 \},
\end{equation}
where
\[ \Sigma(\mathcal{V}_i):= \bigcup_{g \in G} \bigcup^{\infty}_{j=0} \bigcup^{r_i}_{k=1} g \cdot \theta_i(\gamma_k)^j \cdot
g^{-1}.
\]
\end{defin}

\begin{defin} Let $\tau_i:=(g'_i \mid m_{1,i}, \dots , m_{r_i,i})$ for $i=1,2$ be two types.  An \emph{unmixed ramification structure} of type $(\tau_1,\tau_2)$
for a finite group $G$, is a
pair $(\mathcal{V}_1,\mathcal{V}_2)$
of disjoint systems of generators of $G$, whose types are
$\tau_i$, and they satisfy:
\begin{equation}\label{eq.Rim.Hur.Condition}
\mathbb{Z} \ni \frac{|G|
(2g'_i-2+\sum^{r_i}_{l=1}(1-\frac{1}{m_{i,l}}))}{2}+1 \geq 2,
\end{equation}
for $i=1,2$.
\end{defin}
We shall denote by $\mathcal{U}(G;\tau_1,\tau_2)$ the set of all pairs $(\mathcal{V}_1,\mathcal{V}_2)$ of disjoint
systems of generators of unordered type $(\tau_1,\tau_2)$. Here \emph{unordered type} $\tau$ means that there is a permutation $\sigma \in \mathfrak{S}_r$ such that:
${\rm ord}(c_{1}) = m_{\sigma(1)},\dots, {\rm ord}(c_r) = m_{\sigma(r)}$.
We obtain that the datum of a surface isogenous to a higher
product of unmixed type $S=(C_1 \times C_2)/G$  is determined,
looking at the monodromy of each covering of $C_i/G$, by the
datum of a finite group $G$ together with an unmixed ramification
structure. The condition \eqref{eq.sigmasetcond} ensures that the
action of $G$ on the product of the two curves $C_1 \times C_2$ is
free. We remark here that this can be specialized to $C_i/G \cong
\mathbb{P}^1$ in order to obtain regular surfaces isogenous to a product. In this case condition \eqref{eq.Rim.Hur.Condition} is automaticaly satisfied, see \cite[Lemma 2.4]{GP}.
Moreover, we can also ask $r_i=3$, and therefore we obtain
Beauville surfaces, in this case the ramification structure of $G$
is called a \emph{Beauville ramification structure}.

\begin{rem} Note that a group $G$ and an unmixed ramification structure
(or equivalently a Beauville structure) determine the main
invariants of the surface $S$. Indeed, by \eqref{eq.chi.isot.fib} and~\eqref{eq.RiemHurw} we obtain:
\begin{equation}\label{eq.pginfty}
4\chi(S)=|G|\cdot\left({2g'_1-2+\sum^{r_1}_{k=1}(1-\frac{1}{m_{1,k}})}\right)
\cdot\left({2g'_2-2+\sum^{r_2}_{k=1}(1-\frac{1}{m_{2,k}})}\right),
\end{equation}
\end{rem}
and so, in the Beauville case,
\[
    4\chi(S)=4(1+p_g)=|G|(1-\mu_1)(1-\mu_2),
\]
where
\begin{equation}
\label{eq.RHtre} \mu_i:=
\frac{1}{m_{1,i}}+\frac{1}{m_{2,i}}+\frac{1}{m_{3,i}}, \quad
(i=1,2).
\end{equation}

The most important property of surfaces isogenous to a product is their weak rigidity property.
\begin{theo}~\cite[Theorem 3.3, Weak Rigidity Theorem]{cat04}
Let $S=(C_1 \times C_2)/G$ be a surface isogenous to a higher
product of curves. Then every surface with the same
\begin{itemize}
\item topological Euler number and
\item fundamental group
\end{itemize}
is diffeomorphic to $S$. The corresponding  moduli space
$\mathcal{M}^{top}(S) = \mathcal{M}^{diff}(S)$ of surfaces
(orientedly) homeomorphic (resp. diffeomorphic) to $S$ is either
irreducible and connected or consists of two irreducible connected
components exchanged by complex conjugation.
\end{theo}

Thanks to the Weak Rigidity Theorem, we have  that the moduli space
of surfaces isogenous to a product of curves with fixed invariants
--- a finite group $G$ and a type $(\tau_1,\tau_2)$ in the unmixed
case --- consists of a finite number of irreducible connected
components of $\mathcal{M}$. More precisely, let $S$ be a surface
isogenous to a product of curves of unmixed type with group $G$
and a pair of disjoint systems of generators of type
$(\tau_1,\tau_2)$. By~$\eqref{eq.pginfty}$ we have
$\chi(S)=\chi(G,(\tau_1,\tau_2))$, and consequently,
by~\eqref{eq.chi.isot.fib}
$K^2_S=K^2(G,(\tau_1,\tau_2))=8\chi(S)$, and
$e(S)=e(G,(\tau_1,\tau_2))=4\chi(S)$. Moreover the fundamental
group of $S$ fits in the following exact sequence (cf.
\cite{cat00}):
\begin{equation*}
1 \longrightarrow \pi_1(C_1) \times \pi_2(C_2) \longrightarrow \pi_1(S) \longrightarrow G \longrightarrow 1.
\end{equation*}

Let us fix a group $G$ and a type $(\tau_1,\tau_2)$ of an unmixed
ramification structure, and denote by
$\mathcal{M}_{(G,(\tau_1,\tau_2))}$ the moduli space of
isomorphism classes of surfaces isogenous to a product of curves
of unmixed type admitting these data, then it is obviously a
subset of the moduli space
$\mathcal{M}_{K^2(G,(\tau_1,\tau_2)),\chi (G,(\tau_1,\tau_2))}$.
By the Weak Rigidity Theorem, the space
$\mathcal{M}_{(G,(\tau_1,\tau_2))}$ consists of a finite number of
irreducible connected components.

A group theoretical method to count the number of these components
is given in \cite[Theorem 1.3]{BC} in case of surfaces isogenous
to a product of curves of unmixed type with $q=0$ and $G$ abelian. The following theorem is a natural generalization.

\begin{theo}\cite[Theorem 5.7]{P11}\label{Fabmain} Let $S$ be a surface isogenous to a
product of unmixed type. Then we attach to $S$ its finite group
$G$ (up to isomorphism) and the equivalence class ramification structures
$(\mathcal{V}_1,\mathcal{V}_2)$ of type $(\tau_1,\tau_2)$ of $G$,
under the equivalence relation generated by:
\begin{enumerate}
    \item Hurwitz moves and $\Inn(G)$ on $\mathcal{V}_1$,
    \item Hurwitz moves and $\Inn(G)$ on $\mathcal{V}_2$,
    \item simultaneous conjugation of $\mathcal{V}_1$ and
    $\mathcal{V}_2$ by an element $\phi \in \Aut(G)$, i.e., we let
    $(\mathcal{V}_1,\mathcal{V}_2)$ be equivalent to
    $(\phi(\mathcal{V}_1),\phi(\mathcal{V}_2))$.
\end{enumerate}
Then two surfaces $S$ and $S'$ are deformation equivalent if and
only if the corresponding pairs of systems of generators are in
the same equivalence class.
\end{theo}


If we fix a finite group  $G$ and a pair of types
$(\tau_1,\tau_2)$
of an unmixed ramification structure for $G$, counting
the number of connected components of
$\mathcal{M}_{(G,(\tau_1,\tau_2))}$ is then equivalent to the
group theoretical problem of counting the number of classes of
pairs of systems of generators of $G$ of type $(\tau_1,\tau_2)$
under the equivalence relation defined in Theorem \ref{Fabmain}.
This leads also to the
following definition.

\begin{defin}
Denote by $h(G;\tau_1,\tau_2)$ the number of \emph{Hurwitz
components}, namely the number of orbits of
$\mathcal{U}(G;\tau_1,\tau_2)$ under the action of
the group prescribed in Theorem \ref{Fabmain}.
\end{defin}


\section{Connected components of the Moduli Space of Surfaces of General Type}\label{sec.Conne}

We can simplify a lot the discussion of the previous section if we
consider only \emph{regular surfaces}, which will be assumed in the
whole section. In this case, the Hurwitz moves are given only by the $\mathbf{t}_{\sigma_i}$ of Proposition \ref{prop_Hurwitzmoves}, which corresponds to the ones described in Proposition \ref{prop braidaction}. By Theorem \ref{thm_braid} they are given by the
braid group of the sphere on $r$ strands $\mathbf{B}_r$ acting on the generators of the orbifold fundamental group as in Figure 8.

\begin{center}
\includegraphics[totalheight=5 cm, width=15 cm]{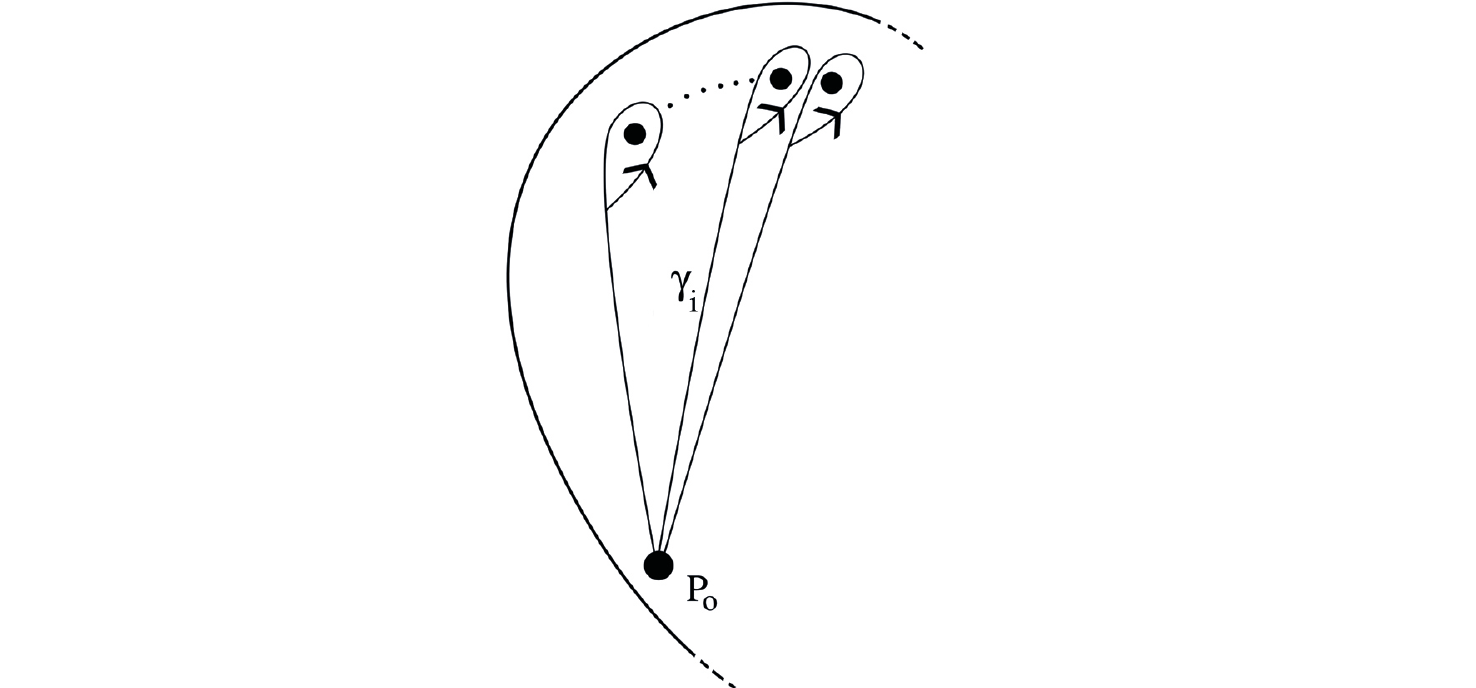} \\
Figure 8.
\end{center}

In addition, we recall

\begin{lem}\cite[Lemma 9.4]{Vol}\label{lem.Vol} The inner automorphism group, $\Inn (G)$, leaves each braid orbit invariant.
\end{lem}

This lemma allows us to use the above Theorem \ref{Fabmain} in the
simplified version without the $\Inn(G)$ action on the system of
generators $\mathcal{V}_i$. Since the two actions of
$\mathbf{B}_r$ and $\Aut(G)$ commute, one gets a double action of
$\mathbf{B}_r \times \Aut(G)$ on the set of $r-$systems of
generators for $G$.

Therefore, for regular surfaces isogenous to a product we have that fixing a finite group  $G$ and a pair of types
$(\tau_1,\tau_2)=(m_{1,1}, \dots , m_{1,r_1}, m_{2,1}, \dots , m_{2,r_2}) $ of an unmixed ramification structure for $G$
counting the number of connected components of
$\mathcal{M}_{(G,(\tau_1,\tau_2))}$ is then equivalent to the
group theoretical problem of counting the number of classes of
pairs of systems of generators of $G$ of type $(\tau_1,\tau_2)$
under the equivalence relation given by the action of
$\mathbf{B}_{r_1} \times \mathbf{B}_{r_2} \times \Aut(G)$. In this
case the number of Hurwitz components $h(G;\tau_1,\tau_2)$ is
given by the number of orbits of $\mathcal{U}(G;\tau_1,\tau_2)$
under the following actions:
\begin{description}
\item[if $\tau_1 \neq \tau_2$] the action of $(\mathbf{B}_{r_1}
\times \mathbf{B}_{r_2}) \times \Aut(G)$, given by:
\[
    ((\gamma_1, \gamma_2), \phi) \cdot (T_1, T_2) := \bigr(\phi(\gamma_1(T_1)),
    \phi(\gamma_2(T_2))\bigl),
\]
where $\gamma_1 \in \mathbf{B}_{r_1}$, $\gamma_2 \in
\mathbf{B}_{r_2}$, $\phi \in \Aut(G)$ and $(T_1,T_2) \in
\mathcal{U}(G;\tau_1,\tau_2)$.

\item[if $\tau_1=\tau_2$] the action of $(\mathbf{B}_{r} \wr
\ZZ/2\ZZ) \times \Aut(G)$, where $\ZZ/2\ZZ$ acts on $(T_1,T_2)$ by
exchanging the two factors.
\end{description}

In case of Beauville surfaces we define $h$ as above substituting
$r_1$ and $r_2$ with $3$.

\begin{prop}\label{prop.asy.comp} Fix $r_1$ and $r_2$ in $\mathbb{N}$.
Let $\{G_n\}^{\infty}_{n=1}$ be a family of finite groups, which
admit an unmixed ramification structure of size $(r_1,r_2)$. Let
$\tau_{n,1} = (m_{n,1,1},\dots,m_{n,1,r_1})$ and $\tau_{n,2} =
(m_{n,2,1},\dots,m_{n,2,r_2})$ be sequences of types
$(\tau_{n,1},\tau_{n,2})$ of unmixed ramification structures for
$G_n$, and $\{S_n\}^{\infty}_{n=1}$ be the family of surfaces
isogenous to product with $q=0$ admitting the given data, then as
$|G_n| \stackrel{n \rightarrow \infty}{\longrightarrow} \infty$ :
\begin{enumerate}\renewcommand{\theenumi}{\it \roman{enumi}}
    \item $\chi(S_n) = \Theta(|G_n|)$.
    \item $h(G_n; \tau_{n,1}, \tau_{n,2}) = O(\chi(S_n)^{r_1+r_2-2})$.
\end{enumerate}
\end{prop}
\begin{proof}
\begin{enumerate}\renewcommand{\theenumi}{\it \roman{enumi}}
\item  Note that, for $i=1,2$,
\[
  \frac{1}{42} \leq -2 + \sum_{j=1}^{r_i} \bigl( 1-\frac{1}{m_{n,i,j}}
  \bigr) \leq r_i-2.
\]
Indeed, for $r_i=3$, the minimal value for $(1-\mu_i)$ is $1/42$.
For $r_i=4$, the minimal value for $\bigl( -2 + \sum_{j=1}^{r_i}
\bigl( 1-\frac{1}{m_{n,i,j}} \bigr) \bigr)$ is $1/6$, and when
$r_i \geq 5$, this value is at least $1/2$.

Now, by Equation~\eqref{eq.pginfty},
\[
    4\chi(S_n) = |G_n| \cdot \left(-2 + \sum_{j=1}^{r_1} \bigl( 1-\frac{1}{m_{n,1,j}}
  \bigr)\right) \cdot \left( -2 + \sum_{j=1}^{r_2} \bigl( 1-\frac{1}{m_{n,2,j}}
  \bigr) \right),
\]
hence
\[
    \frac{|G_n|}{4 \cdot 42^2} \leq \chi(S_n) \leq \frac{(r_1-2)(r_2-2)|G_n|}{4}.
\]

\item For $i=1,2$, any $r_i-$system of generators
$\mathcal{V}_{n,i}$ contains at most $r_i-1$ independent elements
of $G_n$. Thus, the size of the set of all unordered pairs of type
$(\tau_{n,1},\tau_{n,2})$ is bounded from above, by
\[
    |\mathcal{U}(G_n;\tau_{n,1},\tau_{n,2})| \leq |G_n|^{r_1+r_2-2},
\]
and so, the number of connected components is bounded from above
by
\[
    h(G_n;\tau_{n,1},\tau_{n,2}) \leq |G_n|^{r_1+r_2-2}.
\]
Now, the result follows from (1).
\end{enumerate}
\end{proof}
Therefore we cannot expect more than a polynomial growth in the number of connected components of surfaces of general type if we count only regular surfaces isogenous to a product.
Together with Shelly Garion we investigated the asymptotic growth of
the number of connected components of the moduli space of surfaces
of general type by using the group theoretical methods described
above for regular surfaces isogenous to a product. Our first
results already appeared on the ArXiv in 2009 (see arXiv:0910.5402v1). The following are
some of the results contained in \cite{GP}.
\begin{notation} Denote:
\begin{itemize}

\item $h(n)=\Omega(g(n))$, if $h(n) \geq cg(n)$ for some positive
constant $c$, as $n \to \infty$.

\item $h(n)=\Theta(g(n))$, if $c_1g(n) \leq h(n) \leq c_2g(n)$ for
some positive constants $c_1,c_2$, as $n \to \infty$.
\end{itemize}
\end{notation}

\begin{theo}\label{thm.moduli.An}
Let $\tau_1=(m_{1,1}, \dots ,m_{1,r_1})$ and $\tau_2=(m_{2,1}, \dots ,m_{2,r_2})$ be two sequences of natural numbers such
that $m_{k,i} \geq 2$ and $\sum_{i=1}^{r_k}(1-1/m_{k,i}) > 2$ for
$k=1,2$. Let $h(\mathfrak{A}_n, \tau_1, \tau_2)$ be the number
of connected components of the moduli space of surfaces isogenous to a product with $q=0$, with group the alternating group $\mathfrak{A}_n$, and with type
$(\tau_1,\tau_2)$. Then
\[
    (a) \quad h(\mathfrak{A}_n, \tau_1, \tau_2) = \Omega(n^{r_1+r_2}),
\]
and moreover,
\[
    (b) \quad h(\mathfrak{A}_n, \tau_1, \tau_2) =
    \Omega\bigl(\bigl(\log(\chi)\bigr)^{r_1+r_2-\epsilon}\bigr).
\]
where $0< \epsilon \in \mathbb{R}$.
\end{theo}

\begin{theo}\label{thm.moduli.Sn}
Let $\tau_1 = (m_{1,1},\dots,m_{1,r_1})$ and $\tau_2 =
(m_{2,1},\dots,m_{2,r_2})$ be two sequences of natural numbers such
that $m_{k,i} \geq 2$, at least two of $(m_{k,1},\dots,m_{ k,r_k})$
are even and $\sum_{i=1}^{r_k}(1-1/m_{k,i}) > 2$, for $k=1,2$. Let $h(\mathfrak{S}_n,
\tau_1, \tau_2)$ be the number
of connected components of the moduli space of surfaces isogenous to a product with $q=0$, with group the symmetric group
$\mathfrak{S}_n$, and with type $(\tau_1,\tau_2)$. Then
\[
    (a)\quad h(\mathfrak{S}_n, \tau_1, \tau_2) = \Omega(n^{r_1+r_2}),
\]
and moreover,
\[
    (b) \quad h(\mathfrak{S}_n, \tau_1, \tau_2) =
    \Omega\bigl(\bigl(\log(\chi)\bigr)^{r_1+r_2-\epsilon}\bigr).
\]
where $0< \epsilon \in \mathbb{R}$.
\end{theo}

The proofs of part (a) of both Theorems are presented in \cite[Section 3.2]{GP}, and are based on results of Liebeck and Shalev \cite{LS04}. The proofs of part (b)
of both theorems appear in \cite[Section 2]{GP}.

We can specialize the results above to Beauville surfaces. Recall
that a triple $(r,s,t)\in \mathbb{N}^3$ is said to be hyperbolic if
\[ \frac{1}{r}+\frac{1}{s}+\frac{1}{t}<1.
\]

\begin{cor}\label{cor.moduli.Beu.An}
Let $\tau_1=(r_1,s_1,t_1)$ and $\tau_2=(r_2,s_2,t_2)$ be two
hyperbolic types and let $h(\mathfrak{A}_n, \tau_1, \tau_2)$ be the number of
Beauville surfaces with group $\mathfrak{A}_n$ and with types
$(\tau_1,\tau_2)$. Then
\[
    (a) \quad h(\mathfrak{A}_n, \tau_1, \tau_2) = \Omega(n^6),
\]
and moreover,
\[
    (b) \quad h(\mathfrak{A}_n, \tau_1, \tau_2) =
    \Omega\bigl(\bigl(\log(\chi)\bigr)^{6-\epsilon}\bigr).
\]
where $0< \epsilon \in \mathbb{R}$.
\end{cor}

\begin{cor}\label{cor.moduli.Beu.Sn}
Let $\tau_1=(r_1,s_1,t_1)$ and $\tau_2=(r_2,s_2,t_2)$ be two
hyperbolic types, assume that at least two of $(r_1,s_1,t_1)$ are
even and at least two of $(r_2,s_2,t_2)$ are even, and let $h(\mathfrak{S}_n,
\tau_1, \tau_2)$ be the number of Beauville surfaces with group
$\mathfrak{S}_n$ and with types $(\tau_1,\tau_2)$. Then
\[
    (a)\quad h(\mathfrak{S}_n, \tau_1, \tau_2) = \Omega(n^6),
\]
and moreover,
\[
    (b) \quad h(\mathfrak{S}_n, \tau_1, \tau_2) =
    \Omega\bigl(\bigl(\log(\chi)\bigr)^{6-\epsilon}\bigr).
\]
where $0< \epsilon \in \mathbb{R}$.
\end{cor}

The situation is more interesting for abelian groups. We have the following results which assure the existence of ramification structure for abelian group.

\begin{theo}\label{thm.unmixed.abelian}
Let $G$ be an abelian group, given as
\[
    G \cong \mathbb{Z}/{n_1}\mathbb{Z} \times \dots \times
    \mathbb{Z}/{n_t}\mathbb{Z},
\]
where $n_1 \mid \dots \mid n_t$. For a prime $p$, denote by $l_i(p)$
the largest power of $p$ which divides $n_i$ (for $1\leq i \leq t$).

Let $r_1,r_2 \geq 3$, then $G$ admits an unmixed ramification
structure of size $(r_1,r_2)$ if and only if the following
conditions hold:

\begin{itemize}
\item $r_1,r_2 \geq t+1$;
\item $n_t=n_{t-1}$;
\item If $l_{t-1}(3)>l_{t-2}(3)$ then $r_1,r_2\geq 4$;
\item $l_{t-1}(2)=l_{t-2}(2)$;
\item If $l_{t-2}(2)>l_{t-3}(2)$ then $r_1,r_2\geq 5$ and $r_1,r_2$
are not both odd.
\end{itemize}

\end{theo}

This theorem is proved in \cite[Section 3.4]{GP}. Once we
established the existence of ramification structures, we can put
the surfaces in sequences and compute the  number of connected
components of the  moduli space. More precisely, the following holds.

\begin{theo}\label{theo.poly.large}
Let $\{S_p\}$ be the family of surfaces isogenous to a product
with $q=0$ with group $G_p:=(\mathbb{Z}/p \mathbb{Z})^r$ admitting
ramification structure of type $\tau_p = (p,\dots, p)$ ($p$
appears $(r+1)-$times) where $p$ is prime. If we denote by
$h(G_p;\tau_p,\tau_p)$ the number of connected components of the
moduli space of isomorphism classes of surfaces isogenous to a
product with $q=0$ admitting these data, then
\[
h(G_p;\tau_p,\tau_p)=\Theta(\chi^{r}(S_p)).
\]
 \end{theo}

Therefore, there exist families of surfaces such that the degree
of the polynomial $h$ in $\chi$ (and so in $K^2$) can be
arbitrarily large. The proof of this theorem appears in
\cite[Section 2]{GP}. Notice that not only the number of connected
components increases, but also their sizes. Indeed, we see that
$2r-6$ is the dimension of these connected components.

Again we can specialize the results for Beauville surfaces. 

\begin{cor}\label{cor.poly.beau}
Let $\{S_p\}$ be the family of Beauville surfaces
$G_p:=(\mathbb{Z}/p \mathbb{Z})^2$ admitting ramification structure of type $\tau_p =
(p,p,p)$ where $p \geq 5$ is prime. If we denote by $h(G_p;\tau_p,\tau_p)$ the number of Beauville surfaces admitting these data. Then
\[
h(G_p;\tau_p,\tau_p)=\Theta(\chi^{2}(S_p)).
\]
\end{cor}

\proof Let $(x_1,x_2;y_1,y_2)$ be an unmixed Beauville structure
for $G$. Since $x_1, x_2$ are generators of $G$, they are a basis,
and without loss of generality $x_1, x_2$ are the standard basis
$x_1=(1,0)$, $x_2=(0,1)$. Now, let $y_1=(a,b)$, $y_2=(c,d)$, then
the condition \eqref{eq.sigmasetcond} means that any pair of the
six vectors yield a basis of $G$, implying that $a,b,c,d$ must
satisfy the following conditions
\begin{equation}\label{eq.sol.Zp}
    a-b, a+c, c-d, b+d, a+c-b-d, ad-bc \in U
\end{equation}

Moreover, the number $N_p$ of quadruples $(a,b,c,d)$
satisfy~\eqref{eq.sol.Zp} is $N_p=(p-1)(p-2)(p-3)(p-4)$. The pairs
$\bigl((1,0),(0,1);(a,b),(c,d)\bigr)$, where $a,b,c,d$
satisfy~\eqref{eq.sol.Zp}, are exactly the representatives for the
$\Aut(G)-$orbits in the set $\mathcal{U}(G;\tau,\tau)$.

Now, one should consider the action of ${\bf B}_3 \times {\bf B}_3$ on
$\mathcal{U}(G;\tau,\tau)$, which is equivalent to the action of
$\mathfrak{S}_3 \times \mathfrak{S}_3$, since $G$ is abelian. The action of $\mathfrak{S}_3$ on the
second component is obvious (there are $6$ permutations), and the
action of $\mathfrak{S}_3$ on the first component can be translated to an
equivalent $\Aut(G)-$action, given by multiplication in one of the
six matrices:
\[
\left(%
\begin{array}{cc}
  1 & 0 \\
  0 & 1 \\
\end{array}%
\right),
\left(%
\begin{array}{cc}
  0 & 1 \\
  1 & 0 \\
\end{array}%
\right),
\left(%
\begin{array}{cc}
  -1 & 0 \\
  -1 & 1 \\
\end{array}%
\right),
\left(%
\begin{array}{cc}
  1 & -1 \\
  0 & -1 \\
\end{array}%
\right),
\left(%
\begin{array}{cc}
  -1 & 1 \\
  -1 & 0 \\
\end{array}%
\right),
\left(%
\begin{array}{cc}
  0 & -1 \\
  1 & -1 \\
\end{array}%
\right),
\]
yielding an equivalent representative.

Therefore, the action of $\mathfrak{S}_3$ on the second component yields
orbits of length $6$, and the action of $\mathfrak{S}_3$ on the first
component connects them together, and gives orbits of sizes from
$6$ to $36$. Moreover since one can exchange the vector $(x_1,x_2)$ with the vector $(y_1,y_2)$ we get \[
N_p/72 \leq h \leq N_p/6.
\]
By Proposition
\ref{prop.asy.comp}, we have as $p \rightarrow \infty$:
 \[ \chi(S_p) = \Theta(p^2),
 \]
 while by the above computation we have
 \[ h(G_p;\tau_p,\tau_p)=\Theta(p^{4}).
 \]
Therefore
\[
h(G_p;\tau_p,\tau_p)=\Theta(\chi^{2}(S_p)).
\]
\endproof
 \begin{cor}\label{cor.ZnZ.n}
Let $n$ be an integer such that $(n,6)=1$. The number $h=h(G;\tau,\tau)$,
where $\tau=(n,n,n)$, of Hurwitz components for
$G=(\mathbb{Z}/n\mathbb{Z})^2$, where $n=p_1^{k_1}\cdot\ldots\cdot
p_t^{k_t}$, satisfies
\[
N_n/72 \leq h \leq N_n/6,
\]
where $N_n=\prod_{i=1}^t p_i^{4k_i-4}(p_i-1)(p_i-2)(p_i-3)(p_i-4)$.
\end{cor}

\begin{rem}
Notice that if $n$ is divisible by the first $l$ primes $p_i\geq5$ then since:
\[ \lim_{l \rightarrow \infty} \prod_i(1 - \frac{1}{p_i})=0
\]
we have $N_n/n^4 \rightarrow 0$ as $l \rightarrow \infty$.
\end{rem}
In \cite{GJT} the authors give an explicit formula for the
number of isomorphism classes of Beauville surfaces $\Theta(n)$,
which we now explain. We shall keep the notation of \cite{GJT}.
Define the following functions for $n,e \in \mathbb{N}$ and $p$
prime.
\[ \Theta_1(n):=n^4\prod_{p|n}(1-\frac{1}{p})(1-\frac{2}{p})(1-\frac{3}{p})(1-\frac{4}{p});
\]
\[ \Theta_2(p^e):= \left\{
\begin{array}{rl}
p^{2e}(1-\frac{1}{p})(1-\frac{2}{p}) & if \quad p \equiv 1 \mod 4, \\
p^{2e}(1-\frac{1}{p})(1-\frac{4}{p}) & if \quad p \equiv 3 \mod 4; \\
\end{array}
\right.
\]
\[ \Theta_3(p^e):= p^{2e}(1-\frac{3}{p})(1-\frac{5}{p});
\]
\[ \Theta_4(p^e):= \left\{
\begin{array}{rl}
0 & if \quad p \equiv -1 \mod 3, \\
2 & if \quad p \equiv 1 \mod 3. \\
\end{array}
\right.
\]
\begin{theo}\cite[Theorem 2]{GJT} Let $n=p_1^{e_1}\cdot\ldots\cdot
p_t^{e_t}$ be an integer such that $(n,6)=1$. Then the number of isomorphism classes of Beauville surfaces with group $G=(\mathbb{Z}/n\mathbb{Z})^2$ is
\[ \Theta(n):=\frac{1}{72}\Big(\Theta_1(n)+4\prod^{t}_{i=1}\Theta_2(p_i^{e_i}) + 6\prod^{t}_{i=1}\Theta_3(p_i^{e_i})+12\prod^{t}_{i=1}\Theta_4(p_i^{e_i}) \Big).
\]
\end{theo}

The case of \emph{irregular} surfaces has not been so intensively investigated.
Indeed, counting  the number of connected components of the moduli
space of surfaces isogenous to a product with fixed data can be
more complicated. We have to consider Theorem
\ref{Fabmain} together with the Hurwitz moves described in
Proposition \ref{prop_Hurwitzmoves}. Indeed, this procedure was applied in
two specific cases: For surfaces isogenous to a product with
$p_g=q=1$ or $2$ in \cite{pol1} and in \cite{P11}. In both cases it was used a \verb|GAP4| script which can be found in \cite{P10}. 
It would be interesting to consider all the surfaces isogenous to a product with fixed invariants $K^2$ and $\chi$ and count the number of connected components of the moduli space of surfaces of general type they give. Indeed for $K^2=8$ and $\chi=1$ we can count already 94 connected componens.  Respectively: 1 component with $p_g=q=4$, 1 with $p_g=q=3$, 27 with $p_g=q=2$, 52 with $p_g=q=1$ and 13 with $p_g=q=0$.

In this paper we consider only \textit{unmixed} surfaces isogenous to a product, nevertheless similar methods could be applied also for the \textit{mixed} case. Indeed, there are already works in these direction, see \cite{FP13}.

%

\bigskip
\bigskip


\end{document}